\documentclass[final,reqno]{amsart}
\usepackage{amsmath,amsthm,amsfonts,amssymb}
\usepackage{tikz}
\usepackage{soul}
\usepackage{hyperref}

\newcommand{\bs}{\backslash}

\newcommand{\N}{\mathbb{N}}
\newcommand{\R}{\mathbb{R}}
\newcommand{\C}{\mathbb{C}}

\newcommand{\cl}{\operatorname{cl}}
\newcommand{\inter}{\operatorname{int}}
\newcommand{\geA}{\ge_{\mathcal{A}}}
\newcommand{\ngeA}{\ngeq_{\mathcal{A}}}
\newcommand{\gA}{>_{\mathcal{A}}}
\newcommand{\lA}{<_{\mathcal{A}}}

\theoremstyle{plain}
\newtheorem{theorem}{Theorem}[section]

\newtheorem{corollary}[theorem]{Corollary}
\newtheorem{lemma}[theorem]{Lemma}

\theoremstyle{definition}
\newtheorem{definition}[theorem]{Definition}
\newtheorem{remark}[theorem]{Remark}
\newtheorem{example}[theorem]{Example}
\newtheorem{question}[theorem]{Question}

\numberwithin{equation}{section}

\begin{document}

\title[The Berger-Wang formula for maps on cones]{The Berger-Wang formula for order-preserving homogeneous maps on cones}
\author[B. Lins, A. Peperko]{Brian Lins, Aljo\v{s}a Peperko}
\date{}
\address{Brian Lins, Hampden-Sydney College}
\email{blins@hsc.edu}
\address{Aljo\v{s}a Peperko, University of Ljubljana}
\email{Aljosa.Peperko@fs.uni-lj.si}

%%% TO DO: Add subject classifications and keywords.  Maybe add a short bit about Barabanov in the Barabanov norms section.  Maybe swap the order to that continuity comes before boundedness. 
% ALSO: Need to point out that the notion of irreducibility for a family defined in DeiddaGuglielmiTudisco25 is not equivalent to our notion, but actually appears to be closer to primitivity.  

\begin{abstract}
We prove that the joint spectral radius and generalized spectral radius are equal for any bounded, equicontinuous family of order-preserving, homogeneous maps on a polyhedral cone.  We also consider conditions which guarantee that the semigroup generated by a family of order-preserving, homogeneous maps is bounded when its generalized spectral radius $r(\mathcal{A}) = 1$.  Finally, we extend the notions of joint and generalized spectral subradii to the setting of homogeneous maps on wedges.  
\end{abstract}

\maketitle

Rota and Strang \cite{RotaStrang60} first introduced the joint spectral radius for pairs of $n$-by-$n$ matrices.  Later Daubechies and Lagarias \cite{DaubechiesLagarias92} introduced the generalized spectral radius and conjectured that the generalized spectral radius and the joint spectral radius are equal for any finite set of matrices.  This was proved by Berger and Wang \cite{BergerWang92} not just for finite sets of matrices, but also bounded sets.  

Recently there has been interest in nonlinear generalizations of these results.  The joint and generalized spectral radius have been studied in the max algebra setting \cite{Peperko08, Peperko11, Lur06, LurYang09, GMW18}.  More recently \cite{AkianGaubertMarchesini25} introduced the competitive spectral radius for pairs of families of nonexpansive mappings.  The joint spectral radius we consider here can be seen as a one-player analogue of the competitive spectral radius.  We are not the first to consider the joint and generalized spectral radii for families of order-preserving, homogeneous maps on cones.  Motivated in part by neural networks, \cite{DeiddaGuglielmiTudisco25} have generalized the joint spectral radius to order-preserving subhomogeneous maps on cones. 
They also consider both the joint and generalized spectral radii of order-preserving homogeneous maps. In \cite[Corollary 5.22]{DeiddaGuglielmiTudisco25} they prove that the joint and generalized spectral radii are equal for families of continuous, order-preserving, homogeneous maps on a closed cone in a finite dimensional normed space if the maps all have eigenvectors in a compact subset of the interior of the cone.  

The main result of this paper is to prove the equivalence of the joint and generalized spectral radii for any bounded, equicontinuous family of order-preserving, homogeneous maps on a polyhedral cone.  As a corollary, we prove the continuity of the joint spectral radius for compact families with respect to the Hausdorff metric in Section \ref{sec:continuity}.  In Section \ref{sec:bounded} we introduce irreducibility and primitivity conditions which can guarantee that the semigroup generated by a family of order-preserving, homogeneous maps is bounded.  In Section \ref{sec:norms} we prove the existence of extremal Barabanov norms for irreducible families of subadditive, order-preserving, homogeneous maps on the standard cone in $\R^n$.  Finally, in Section \ref{sec:subradii} we extend the notions of joint and generalized spectral subradii to homogeneous maps on wedges.

\section{Preliminaries} \label{sec:prelim}
 
Let $(X,\|\cdot\|)$ be a real normed space. For any set $S$, let $\inter S$ denote the interior and $\cl S$ denote the closure of $S$. A convex set $W \subseteq X$ is a \emph{wedge} if $\lambda W \subseteq W$ for all real $\lambda > 0$.  A wedge $K \subset X$ is a \emph{cone} if $K \cap (-K) = \{0\}$. A cone is \emph{solid} if it has nonempty interior and \emph{polyhedral} if it has a finite number of extreme rays.  Any closed cone $K$ defines a partial order on $X$ where $x \ge y$ if and only if $x-y \in K$. We write $x \gg y$ if $x - y \in \inter K$. A closed cone $K$ is \emph{normal} if there is a constant $M > 0$ such that $\|x\| \le M \| y \|$ whenever $x, y \in K$ and $x \le y$. Note that any closed cone in a finite dimensional normed space is normal \cite[Lemma 1.2.5]{LemmensNussbaum}.  A norm $\|\cdot \|$ is \emph{monotone} if $\|x\| \le \|y\|$ whenever $0 \le x \le y$. 

A closed cone $K$ has \emph{condition $\mathsf{G}$} if for every $x \in K$ and $c < 1$, there is a $\delta > 0$ such that $y \ge c x$ for all $y \in K$ with $\|x - y\| < \delta$.  A closed cone is polyhedral if and only if is has condition $\mathsf{G}$ \cite[Lemma 3.3, Lemma 5.1.4]{BurbanksNussbaumSparrow03, LemmensNussbaum}.

Given a domain $D \subseteq X$, we say that a map $f:D \rightarrow X$ is \emph{order-preserving} if $f(x) \ge f(y)$ whenever $x \ge y$. We say that $f$ is \emph{(positively) homogeneous} if $f(cx) = c f(x)$ for all $c > 0$ and $x \in D$. 
%If $f: \inter K \rightarrow \inter K$ is an order-preserving, homogeneous map on the interior of a closed polyhedral cone, then $f$ has a unique continuous extension to all of $K$ that is order-preserving and homogeneous \cite[Corollary 2.1.4 and Theorem 5.1.5]{LemmensNussbaum}. %CHECK THIS!

Two elements $x, y$ in a closed cone $K$ are \emph{comparable} if there are positive constants $\alpha, \beta$ such that $\alpha x \le y \le \beta x$.  Comparability is an equivalence relation on $K$, and the equivalence classes are called the \emph{parts} of $K$. We consider the parts $\inter K$ and $\{0\}$ to be trivial, all other parts are non-trivial. Note that a closed cone is polyhedral if and only if it has a finite number of parts. The closures of the parts are the faces of the polyhedral cone.

An important polyhedral cone is the \emph{standard cone} in $\R^n$:
$$\R^n_{\ge 0} := \{x \in \R^n : x_i \ge 0 \text{ for all } i \in [n] \}.$$
The interior of the standard cone is 
$$\R^n_{> 0} := \{x \in \R^n : x_i > 0 \text{ for all } i \in [n] \}.$$
Of particular interest are \emph{multiplicatively topical} maps which are continuous, order-preserving, homogeneous maps $f: \R^n_{\ge 0} \rightarrow \R^n_{\ge 0}$ such that $f(\R^n_{>0}) \subseteq \R^n_{>0}$. When working with $\R^n$, we let $e_1, \ldots, e_n \in \R^n$ denote the standard basis vectors and $\mathbf{1} := \sum_{i = 1}^n e_i$ is the vector with all entries equal to one. %When working with $\R^n$, we will assume that the norm on $\R^n$ is \emph{monotone}, that is $\|x\| \le \|y\|$ whenever $0 \le x \le y$.  All of the standard $\ell^p$-norms with $1 \le p \le \infty$ are monotone.  

Let $X$ be real normed vector space and $W\subset X$ a non-zero wedge.
The norm of a homogeneous map $f:W \rightarrow X$ on a wedge $W \subseteq X$ is 
$$\|f\|_W := \sup \{\|f(x)\| : x \in W, \|x\| = 1 \}.$$ 
When the domain of the function is understood, we omit the subscript and just write $\|f\|$.  A homogeneous map $f$ is \emph{bounded} if $\|f\| < \infty$.

%Let $X$ be a normed space and $C \subset X$ a non-zero cone. 
Let $f:W \to W$ be homogeneous and bounded.
It is easy to see that
$\|f^{m+n}\|\le \|f^m\|\cdot \|f^n\|$ for all $m,n\in\mathbb{N}$. By Fekete's subadditivity lemma it follows that the limit $\lim_{n\to\infty}\|f^n\|^{1/n}$ exists and is equal to $\inf_{n>0} \|f^n\|^{1/n}$. The \emph{(Bonsall) cone spectral radius of $f$} is the limit
% $r(f)=\lim_{n\to\infty}\|f^n\|^{1/n}= \inf _{n\in \mathbb{N}}\|f^n\|^{1/n} $ is called the (Bonsall) cone spectral radius of $f$. Observe that all the above properties also hold if $K\subset X$ is a non-zero set such that $tK \subset K $ for all $t\ge 0 $.
%The \emph{cone spectral radius} of a continuous, order-preserving, homogeneous map $f:K \rightarrow K$ is 
$$r(f) := \lim_{n \rightarrow \infty} \|f^n\|^{1/n} = \inf_{n > 0} \|f^n\|^{1/n}.$$
When $K$ is a normal closed cone with nonempty interior in a real Banach space and $f:K \rightarrow K$ is continuous, homogeneous, and order-preserving, it is known that for any $x \in \inter K$, $r(f) = \lim_{n \rightarrow \infty} \|f^n(x) \|^{1/n}$ \cite[Theorem 2.2]{MalletParetNussbaum02}.  When $K$ is a closed cone in a finite dimensional real normed space and $f: K \rightarrow K$ is continuous, order-preserving, and homogeneous, there is always a nonzero eigenvector $x \in K$ with eigenvalue $r(f)$, that is, $f(x) = r(f) x$ \cite[Corollary 5.4.2]{LemmensNussbaum}.  

%Let $K$ be a closed cone in a finite dimensional Banach space $X$. For any $M > 0$, let $\Omega_M := \{x \in K : \|x \| \le M \}$. Note that a family $\mathcal{A}$ of homogeneous functions $f:K \rightarrow K$ is equicontinuous for every $x \in K$ if and only if the family is equicontinuous for every $x \in \Omega_M$. 
%Let $\Sigma = \{x \in \R^n_{\ge 0} : \|x \| = 1 \}$.  

Let $\mathcal{A}$ be a set of homogeneous maps on a domain $D \subseteq X$.  The family $\mathcal{A}$ is \emph{bounded} if $\{ \|f\| : f \in \mathcal{A} \}$ is bounded and $\mathcal{A}$ is \emph{equicontinuous} if it is pointwise equicontinuous on $D$.  %If the domain is a closed cone $K$ in a finite dimensional space $X$, then $\mathcal{A}$ is equicontinuous if and only if it is uniformly equicontinuous on the compact set 
%$$\Sigma := \{x \in K : \|x\| = 1 \}.$$ 

If $D \subset X$ is compact and $X$ is a real Banach space, then any continuous, homogeneous map $f: D \rightarrow X$ is also an element of the Banach space $C( D, X)$ which consists of continuous functions from $D$ into $X$.  The norm on $C(D, X)$ is 
$$\|f\| := \sup_{x \in D} \|f(x) \|.$$
If $K \subset X$ is a closed cone in a finite dimensional normed space, then by the Arzel\`{a}-Ascoli theorem, a closed and bounded family $\mathcal{A}$ of continuous, homogeneous functions $f: K \rightarrow K$ is equicontinuous if and only if it is compact as a family of maps in  $C(\Omega,X)$ for every compact $\Omega \subset K$. 

Let $\mathcal{A}$ be a family of functions $f: Y \rightarrow Z$ and let $\mathcal{B}$ be a family of functions $f:X \rightarrow Y$ for some sets $X, Y, Z$. Then the \emph{composition} of the two families is
$$\mathcal{AB} := \{f \circ g : f \in \mathcal{A}, g \in \mathcal{B} \}.$$
When working with a family $\mathcal{A}$ of maps $f: D \rightarrow D$, we adopt the following notational conventions.  For each $m \in \N$, we let 
$$\mathcal{A}^m := \{ f_m \circ \cdots \circ f_1 : f_i \in \mathcal{A} \}.$$
We use $\mathcal{A}^+$ and $\mathcal{A}^*$ to denote the semigroup and monoid generated by $\mathcal{A}$: 
$$\mathcal{A}^+ := \bigcup_{m > 0} \mathcal{A}^m, \text{ and } \mathcal{A}^* := \bigcup_{m \ge 0} \mathcal{A}^m.$$

The \emph{joint spectral radius} of $\mathcal{A}$ is the quantity $\hat{r}(\mathcal{A})$ defined in the following lemma. 

\begin{lemma} \label{lem:joint}
%Let $X$ be a real normed space and let $W \subseteq X$ be a wedge. 
Let $W$ be a wedge in a real normed space.
For any bounded family $\mathcal{A}$ of homogeneous maps on $W$, the joint spectral radius of $\mathcal{A}$ is
$$\hat{r}(\mathcal{A}) := \inf_{m > 0} \sup_{f \in \mathcal{A}^m} \| f\|^{1/m} = \lim_{m \rightarrow \infty} \sup_{f \in \mathcal{A}^m} \| f\|^{1/m}.$$
\end{lemma}

\begin{proof}
Let $\alpha_m = \sup_{f \in \mathcal{A}^m} \|f\|$.  
For any $m, n \in \N$, observe that 
\begin{align*} 
\alpha_{m+n} &= \sup_{f \in \mathcal{A}^m, g \in \mathcal{A}^n} \sup_{x \in W, \|x\|=1} \|f(g(x))\| \\
&\le \sup_{f \in \mathcal{A}^m, g \in \mathcal{A}^n} \|f\| \|g\| = \alpha_m \alpha_n. 
\end{align*}
Therefore, by Fekete's lemma, the sequence $\alpha_m^{1/m}$ converges, and $\inf_{m > 0} \alpha_m^{1/m} = \lim_{m \rightarrow \infty} \alpha_m^{1/m}$. 
\end{proof}

We also define the \emph{generalized spectral radius} of $\mathcal{A}$ as follows. 

\begin{lemma} \label{lem:generalized}
Let $\mathcal{A}$ be a bounded family of homogeneous maps on a wedge $W$ in a real normed space. The generalized spectral radius of $\mathcal{A}$ is 
$$r(\mathcal{A}):=\limsup_{n \rightarrow \infty} \sup_{f \in \mathcal{A}^n} r(f)^{1/n} = \sup_{m > 0} \sup_{f \in \mathcal{A}^m} r(f)^{1/m}.$$
\end{lemma}

\begin{proof}
Clearly 
$$\limsup_{n \rightarrow \infty} \sup_{f \in \mathcal{A}^n} r(f)^{1/n} \le \sup_{m > 0} \sup_{f \in \mathcal{A}^m} r(f)^{1/m}.$$ 
It is clear from the definition of the cone spectral radius that $r(f^n) = r(f)^n$ for any bounded, homogeneous map $f: W \rightarrow W$.  For any $m, n \in \N$, 
$$\sup_{g \in \mathcal{A}^{mn}} r(g)^{1/(mn)} \ge \sup_{f \in \mathcal{A}^m} r(f^n)^{1/(mn)} = \sup_{f \in \mathcal{A}^m} r(f)^{1/m}.$$
Then for any $m \in \N$, 
$$\limsup_{n \rightarrow \infty} \sup_{f \in \mathcal{A}^n} r(f)^{1/n} \ge \limsup_{n \rightarrow \infty} \sup_{g \in \mathcal{A}^{mn}} r(g)^{1/(mn)} \ge \sup_{f \in \mathcal{A}^m} r(f)^{1/m}.$$
This proves that
$$\limsup_{n \rightarrow \infty} \sup_{f \in \mathcal{A}^n} r(f)^{1/n} = \sup_{m > 0} \sup_{f \in \mathcal{A}^m} r(f)^{1/m}.$$
\end{proof}

In the sequel, we will also need to consider the \emph{partial joint spectral radius}. 

\begin{lemma} \label{lem:partial}
Let $K$ be a closed cone in a finite dimensional real normed space. Let $\mathcal{A}$ be a bounded family of continuous, homogeneous, order-preserving maps on $K$ and let $P$ be a part of $K$. The partial joint spectral radius of $\mathcal{A}$ on $P$ is 
$$\hat{r}(\mathcal{A},P) := \limsup_{m \rightarrow \infty} \sup_{f \in \mathcal{A}^m} \|f\|^{1/m}_P = \limsup_{m \rightarrow \infty} \sup_{f \in \mathcal{A}^m} \|f(x)\|^{1/m}$$
for every $x \in P$.  In particular, if $P = \inter K$, then
$$\hat{r}(\mathcal{A}) = \hat{r}(\mathcal{A}, P) = \lim_{m \rightarrow \infty} \sup_{f \in \mathcal{A}^m} \|f(x)\|^{1/m}$$
for every $x \in \inter K$. 
\end{lemma}

\begin{proof}
Choose any $x \in P$. Since parts are relatively open convex sets \cite[1.2.2]{LemmensNussbaum}, there is an $\epsilon > 0$ small enough so that $x - \epsilon y \in P$ for every $y \in P$ with $\|y\| = 1$.  Then $x \ge \epsilon y$ for every $y \in P$ with $\|y\|=1$.  Since $K$ is a finite dimensional cone, it must be normal which means that there is a constant $M > 0$ such that whenever $0 \le v \le w$, $\|v\| \le M \|w\|$. In particular, 
$$\|f(\epsilon y)\| \le M \|f(x)\|$$
for every $y \in P$ with $\|y \| =1$ and any order-preserving, homogeneous map $f$.  
Therefore 
\begin{equation} \label{eq:partial}
\epsilon\|f\|_P \le M \|f(x)\| \le M \|f\|_P  \|x\|
\end{equation}
for every $f \in \mathcal{A}^+$.  From this, we conclude that 
$$\limsup_{m \rightarrow \infty} \sup_{f \in \mathcal{A}^m} \|f\|^{1/m}_P = \limsup_{m \rightarrow \infty} \sup_{f \in \mathcal{A}^m} \|f(x)\|^{1/m}.$$
If $P = \inter K$, then $\|f\|_P = \|f\|$ for any continuous, order-preserving, homogeneous map on $K$, so $\hat{r}(\mathcal{A}) = \hat{r}(\mathcal{A},P)$ and both sequences $\sup_{f \in \mathcal{A}^m} \|f\|_P^{1/m}$ and $\sup_{f \in \mathcal{A}^m} \|f(x)\|^{1/m}$ converge by \eqref{eq:partial} and Lemma \ref{lem:joint}.  
\end{proof}

We will also need the following minor technical lemma.  
\begin{lemma} \label{lem:composition}
Let $K$ be a closed cone in a finite dimensional normed space $X$. Let $\mathcal{A}$ and $\mathcal{B}$ be bounded equicontinuous families of homogeneous maps $f: K \rightarrow K$.  Then the family $\mathcal{AB}$ is equicontinuous.  
\end{lemma}

\begin{proof}
For any $R \ge 0$, let $B_R := \{x \in X : \|x\| \le R \}$.  
Since $\mathcal{B}$ is bounded, there is a constant $M>0$ such that $g(B_1 \cap K) \subseteq B_M \cap K$ for all $g \in \mathcal{B}$. 
Since $\mathcal{A}$, $\mathcal{B}$ are equicontinuous, they are uniformly equicontinuous on the compact set $B_M \cap K$.
For any $\epsilon > 0$, there exists $\delta > 0$ such that 
$\|f(x) - f(y)\| \le \epsilon$ for any $x, y \in B_M \cap K$ with $\|x-y\| \le \delta$.  
At the same time, there exists $\gamma > 0$ such that 
$\|g(x) - g(y)\| \le \delta$ for any $x, y \in B_1 \cap K$ with $\|x - y\| \le \gamma$.  
Therefore, 
$$\|f(g(x)) - f(g(y)) \| \le \epsilon$$
for any $x, y \in B_1 \cap K$ with $\|x-y\| \le \gamma$. 
This proves that $\mathcal{AB}$ is uniformly equicontinuous on $B_1 \cap K$.  It follows by homogeneity that $\mathcal{AB}$ is equicontinuous on $K$. 
\end{proof}

\section{A Berger-Wang Formula} \label{sec:BW}

\begin{theorem} \label{thm:main}
Suppose that $K$ is a closed polyhedral cone in a finite dimensional normed space $X$. If $\mathcal{A}$ is a bounded, equicontinuous family of order-preserving, homogeneous maps on $K$, then 
$$r(\mathcal{A}) = \hat{r}(\mathcal{A}).$$
\end{theorem}

In order to prove Theorem \ref{thm:main}, we collect some preliminary results below.

\begin{lemma} \label{lem:ineq}
Let $W$ be a wedge in a real normed space.  For any bounded family $\mathcal{A}$ of homogeneous maps on $W$, $r(\mathcal{A}) \le \hat{r}(\mathcal{A})$. 
\end{lemma}

\begin{proof}
For any bounded homogeneous map $f: W \rightarrow W$, $r(f) \le \|f\|$. Therefore 
$$\sup_{f \in \mathcal{A}^m} r(f)^{1/m} \le \sup_{f \in \mathcal{A}^m} \|f\|^{1/m}$$ 
for every $m \in \N$. Then  
$$r(\mathcal{A}) = \limsup_{m \rightarrow \infty} \sup_{f \in \mathcal{A}^m} r(f)^{1/m} \le \lim_{m \rightarrow \infty} \sup_{f \in \mathcal{A}^m} \|f\|^{1/m} = \hat{r}(\mathcal{A}).$$
\end{proof}

\begin{lemma} \label{lem:alpha}
Let $K$ be a closed cone in a real normed space, let $\mathcal{A}$ be a bounded family of homogeneous, order-preserving maps on $K$, and let $m \in \N$. If there exists $f \in  \mathcal{A}^m$,  $x \in K \bs \{0\}$, and $\alpha > 0$ such that $f(x) \ge \alpha^m x$, then $r(\mathcal{A}) \ge \alpha$. 
\end{lemma}

\begin{proof}
Suppose $f\in \mathcal{A}^m$ and $f(x) \ge \alpha^m x$.  Since $f$ is order-preserving and homogeneous, $f^n(x) \ge \alpha^{mn}x$ for every $n \in \N$.  This means that $\|f^n\| \ge \alpha^{mn}$ for all $n \in \N$.  Therefore
$$r(f) = \lim_{n \rightarrow \infty} \|f^n\|^{1/n} \ge \alpha^m.$$  
Then by Lemma \ref{lem:generalized}, $r(\mathcal{A}) \ge r(f)^{1/m} \ge \alpha$. 
\end{proof}

For a bounded family of continuous, order-preserving, homogeneous maps $\mathcal{A}$ on a closed polyhedral cone $K$ we can define a preorder on the parts of $K$ as follows.  We say that $P \geA Q$ if there exists $x \in P$, $y \in Q$, and $f \in \mathcal{A}^*$ such that $f(x) \ge y$. We will write $P \gA Q$ if $P \geA Q$ and $Q \ngeA P$.  Observe that the relation $\geA$ is transitive since if $P \geA Q$ and $Q \geA R$, then there exist $f, g \in \mathcal{A}^*$ and $x \in P$, $y \in Q$, $z \in R$ such that $f(x) \ge y$ and $g(y) \ge z$.  Then $g(f(x)) \ge g(y) \ge z$, so $P \geA R$. 

\begin{lemma} \label{lem:partiallimit}
Let $K$ be a closed polyhedral cone in a finite dimensional normed space and let $\mathcal{A}$ be a bounded family of continuous, order-preserving, homogeneous maps on $K$. Assume that $r(\mathcal{A}) < 1$. Suppose that $x \in K$ and there is an increasing sequence $m_k \in \N$ and maps $f_k \in \mathcal{A}^{m_k}$ such that $\|f_k(x)\| \ge 1$ for every $k \in \N$ and 
$$y = \lim_{k \rightarrow \infty} \frac{f_k(x)}{\|f_k(x)\|}.$$ 
Let $P_x, P_y$ denote the parts of $K$ containing $x$ and $y$, respectively.  Then $P_x \gA P_y$.
\end{lemma}

\begin{proof}
Since $K$ is polyhedral, it has condition $\mathsf{G}$. So for any $c < 1$, there is an $m$ such that
$$f_k(x) \ge f_k(x)/\|f_k(x)\| \ge c y$$ 
for every $k \ge m$.  Therefore $P_x \geA P_y$.  

Suppose that $P_y \geA P_x$ as well.  Then there exist $y' \in P_y$ and $x' \in P_x$ and $g \in \mathcal{A}^n$ for some $n \in \N$ such that $g(y') \ge x'$.  Since $y'$ is comparable to $y$ and $x'$ is comparable to $x$, there is a constant $\epsilon > 0$ such that 
$g(y) \ge \epsilon x$.  Then 
$$\lim_{k \rightarrow \infty} g\left( \frac{f_k(x)}{\|f_k(x)\|} \right) = g(y) \ge \epsilon x.$$
Since $K$ is polyhedral, condition $\mathsf{G}$ guarantees that for any $0 < c < 1$, there is an $m > 0$ such that 
$$g(f_k(x)) \ge g\left( \frac{f_k(x)}{\|f_k(x)\|} \right)  \ge c \epsilon x$$
for all $k \ge m$.  
Since $r(\mathcal{A}) < 1$ we can choose a constant $\alpha$ such that $r(\mathcal{A}) < \alpha < 1$.  If $k$ is large enough, then 
$$g(f_k(x)) \ge c \epsilon x \ge (\alpha^{n + m_k})x.$$
This implies that $r(\mathcal{A}) \ge \alpha$ by Lemma \ref{lem:alpha}, but that is a contradiction. So we conclude that $P_x \gA P_y$.  
\end{proof}

\begin{lemma} \label{lem:boundedpartials}
Suppose that $K$ is a closed polyhedral cone in a finite dimensional normed space. Let $\mathcal{A}$ be a bounded, equicontinuous family of order-preserving, homogeneous maps on $K$.  If $\hat{r}(\mathcal{A}, Q) \le r(\mathcal{A})$ for all parts $Q \lA P$, then $\hat{r}(\mathcal{A},P) \le r(\mathcal{A})$. 
\end{lemma}

\begin{proof}
Suppose that $\hat{r}(\mathcal{A}, P) > r(\mathcal{A})$.  We may scale the functions in $\mathcal{A}$ to assume without loss of generality that $r(\mathcal{A}) < 1$ and $\hat{r}(\mathcal{A}, P) > 1$.  
Fix a constant $r(\mathcal{A}) < \beta < 1$.  By Lemma \ref{lem:partial}, there is an $m \in \N$ large enough so that $\|f\|_Q \le \beta$ for every part $Q \lA P$ and $f \in \mathcal{A}^m$.  

Choose $x \in P$.  Suppose that the set $\mathcal{O}_m := \{f(x) : f \in \mathcal{A}^{mk}, k \in \N \}$ is unbounded.  Then there is a sequence $y_k \in \mathcal{O}_m$ such that for each $y_k$ there exists $f \in \mathcal{A}^m$ with 
$$\|f(y_k)\| \ge \|y_k\| \ge k.$$  
We may assume, by passing to a subsequence that $\lim_{k \rightarrow \infty} y_k/\|y_k\| = z$ exists.  By Lemma \ref{lem:partiallimit} $z$ is contained in a part $Q \lA P$. 

Choose $\epsilon > 0$ small enough so that $\beta + \epsilon < 1$.  Since the family $\mathcal{A}^m$ is equicontinuous by Lemma \ref{lem:composition}, there exists $\delta > 0$ such that $\|f(y) - f(z)\| < \epsilon$ for any $f \in \mathcal{A}^m$ and $y \in K$ with $\|y-z\| < \delta$. In particular, 
$$\|f(y)\| \le \|f(z)\| + \epsilon \le \beta + \epsilon.$$
Since $y_k/\|y_k\|$ converges to $z$, there is a $k$ large enough so that 
$$\left\| f\left( \frac{y_k}{\|y_k\|} \right) \right\| \le \beta + \epsilon.$$
Then 
$$\|f(y_k)\| \le (\beta + \epsilon) \|y_k\| < \|y_k\|$$
for all $f \in \mathcal{A}^m$, which is a contradiction.  So we conclude that $\mathcal{O}_m$ is bounded.  In that case, the set $\{f(x) : f \in \mathcal{A}^+\}$ is also bounded, so $\hat{r}(\mathcal{A}, P) \le 1$ by Lemma \ref{lem:partial}. But that contradicts our assumption that $\hat{r}(\mathcal{A}, P) > 1$, so we conclude that $\hat{r}(\mathcal{A}, P)  \le r(\mathcal{A})$. 
\end{proof}

\begin{proof}[Proof of Theorem \ref{thm:main}]
We have already shown in Lemma \ref{lem:ineq} that $\hat{r}(\mathcal{A}) \ge r(\mathcal{A})$.
We can assume that $K$ has nonempty interior.  Otherwise we can replace $X$ by the subspace $Y$ spanned by $K$. Then $K$ must have nonempty interior in $Y$ since $K$ is convex \cite[Theorem 6.2]{Rockafellar}. 

Suppose that there are parts $P$ such that $\hat{r}(\mathcal{A}, P) > r(\mathcal{A})$.  Since there are only finitely many parts with this property, there must be at least one that is minimal with respect to the relation $\lA$. But that contradicts Lemma \ref{lem:boundedpartials}. So we conclude that $\hat{r}(\mathcal{A}, P) \le r(\mathcal{A})$ for all parts including $\inter K$, and therefore $\hat{r}(\mathcal{A}) = r(\mathcal{A})$ by Lemma \ref{lem:partial}. 
\end{proof}

\begin{example}
Let $\mathcal{A}$ be a family of continuous, order-preserving, homogeneous maps on a closed cone $K$ in a finite dimensional normed space. Unlike families of linear maps, it is possible for $\mathcal{A}$ to be bounded without being equicontinuous.  For example, consider $K = \R^2_{\ge 0}$ and $\mathcal{A} = \{f_k : k > 0\}$ where 
$$f_k(x) = \begin{bmatrix} \min \{x_2, k x_1\} \\ x_2 \end{bmatrix}.$$ 
Then $f_k(e_2) = e_2$ but $f_k(\delta e_1 + e_2) = \mathbf{1}$
%$$f_k\left( \begin{bmatrix} 0 \\ 1 \end{bmatrix} \right)  = \begin{bmatrix} 0 \\ 1 \end{bmatrix}, \text{ but }
%f_k\left( \begin{bmatrix} \delta \\ 1 \end{bmatrix} \right)  = \begin{bmatrix} 1 \\ 1 \end{bmatrix}$$  
whenever $k > 1/\delta$. 
\end{example}

It turns out that even a simple bounded family of continuous, order-preserving, homogeneous maps on a polyhedral cone can fail the Berger-Wang condition in Theorem \ref{thm:main} if the family is not equicontinuous. 

\begin{example}
Consider the family of multiplicatively topical maps $\mathcal{A}$ on $\R^2_{\ge 0}$ defined by
$$f_\lambda(x) = \begin{bmatrix} x_1^\lambda x_2^{(1-\lambda)} \\ \tfrac{1}{2} x_2 \end{bmatrix}, ~ 0 < \lambda < 1.$$
%This family of maps is \emph{multiplicatively convex} which means that $\log \circ f_\lambda \circ \exp$ is a convex function (here $\log$ and $\exp$ represent the entrywise natural logarithm and exponential functions).  

Observe that if $f \in \mathcal{A}^m$, then $f(x)_2 = 2^{-m} x_2$ for any $x \in \R^2_{\ge 0}$. Therefore, the only possible eigenvalues of $f$ are $0$ (which corresponds to the eigenvector $e_1$) and $2^{-m}$ (corresponding to $e_2$). Since $r(f)$ must be an eigenvalue of $f$, it follows that $r(f) = 2^{-m}$, and therefore $r(\mathcal{A}) = \tfrac{1}{2}$ by Lemma \ref{lem:generalized}. 

For any $m$ and $\epsilon > 0$, it is possible to choose $\lambda$ large enough so that $f_\lambda^m(\mathbf{1})_1 > 1-\epsilon$. But then $\sup_{f \in \mathcal{A}^m} \|f\|^{1/m} = 1$ which implies that $\hat{r}(\mathcal{A}) = 1$ by Lemma \ref{lem:joint}. Therefore the generalized spectral radius of $\mathcal{A}$ is strictly smaller than the joint spectral radius of $\mathcal{A}$. 
\end{example}

Let $K$ be a closed cone in a real normed space.  A map $f: K \rightarrow K$ is \emph{subadditive} if $f(x + y) \le f(x) + f(y)$ for all $x, y \in K$.  The following lemma shows that we can remove the equicontinuity assumption from Theorem \ref{thm:main} if the maps are all subadditive. 

%\begin{lemma}
%Let $f: \R^n_{\ge 0} \rightarrow \R^n_{\ge 0}$ be a subadditive multiplicatively topical map.  Then $f$ is multiplicatively convex.  
%\end{lemma}
%
%\begin{proof}
%Since $f$ is homogeneous and subadditive, it is a convex function on $\R^n_{>0}$.  Fix any $x, y \in \R^n_{>0}$ and $0< \lambda < 1$. Let $z \in \R^n_{>0}$ be defined by $z_i = x_i^\lambda y_i^{1 - \lambda}$ for every $i \in [n]$.  Since $f$ is convex, it has a linear functional $\phi$ in its subdifferential set at $z$ such that $f(z) = \phi(z)$ and $f(x) \ge \phi(x)$ for all $x \in \R^n_{>0}$. Since $f$ is order-preserving, $\phi(e_i) \ge 0$ for every standard basis vector $e_i$, $i \in [n]$. In particular, $\phi$ is a nonnegative linear function from $\R^n_{>0}$ to $\R$ and so it is multiplicatively convex \cite[Lemma 4.1]{Lins23}.  Therefore,
%$$f(z)_i = \phi(z)_i \le \phi(x)_i^\lambda \phi(y)_i^{1 -\lambda} \le f(x)_i^{\lambda} f(y)_i^{1- \lambda}$$
%for $i \in [n]$.  This proves that each entry of $f$ is multiplicatively convex, and therefore so is $f$.  
%\end{proof}

\begin{lemma} \label{lem:subEqui}
Let $K$ be a closed polyhedral cone in a finite dimensional normed space.  If $\mathcal{A}$ is a bounded family of subadditive, order-preserving, homogeneous maps $f: K \rightarrow K$, then $\mathcal{A}$ is equicontinuous. 
\end{lemma}

\begin{proof}
We have to show for any $x \in K$ and for any $\epsilon > 0$, there is a $\delta > 0$ such that $\|f(x) - f(y) \| < \epsilon$ for any $f \in \mathcal{A}$ and $y \in K$ such that $\|x - y\| < \delta$.  Since $K$ is polyhedral, it satisfies condition $\mathsf{G}$.  In particular, for any $\gamma > 0$, there exists $\delta > 0$ such that $(1+\gamma) y \ge x$ for any $y \in K$ with $\|x-y \| < \delta$.  

Let $z := (1+\gamma) y - x$ so that $z \in K$ and  
$$(1+ \gamma) f(y) = f(x + z) \le f(x) + f(z).$$
Therefore 
$$0 \le (1+ \gamma) f(y) - f(x) \le f(z).$$
Since $K$ is a closed cone in a finite dimensional space, it is normal and there is a constant $M > 0$ such that 
$$ \|(1+ \gamma) f(y) - f(x)\| \le M \|f(z)\|.$$
By the triangle inequality, 
$$ \|f(y) - f(x)\| \le \|(1+\gamma)f(y) - f(x) \| + \gamma \|f(y)\| \le M \|f(z)\| + \gamma \|f(y)\|.$$
Since $\mathcal{A}$ is bounded, there is an $N > 0$ such that $\|f\| \le N$ for all $f \in \mathcal{A}$.  Then 
$$ \|f(y) - f(x)\| \le M N \|z\| + \gamma N \|y\|.$$
By construction, 
$$\|z\| \le \|(1+ \gamma) y - x\| \le \|y - x\| + \gamma \|y\| \le \delta  + \gamma \|y\|$$
and 
$$\|y\| \le \|x\| + \delta.$$
Therefore 
$$ \|f(y) - f(x)\| \le M N (\delta + \gamma(\|x\| + \delta)) + \gamma N (\|x\| + \delta).$$
Since we can choose both $\gamma$ and $\delta$ arbitrarily small, we can guarantee that $\|f(y) - f(x)\| < \epsilon$ for any $\epsilon >0$.  
\end{proof}

\begin{remark}
Lemma \ref{lem:subEqui} implies that if $K$ is a closed polyhedral cone and $f: K \rightarrow K$ is order-preserving, homogeneous, and subadditive, then $f$ must be continuous.  To see this, simply apply the theorem to a family which only contains the one function $f$. 
\end{remark}

\begin{question} 
The proof of Theorem \ref{thm:main} makes heavy use of the fact that $K$ is polyhedral.  Is the Berger-Wang formula true if $K$ is not polyhedral?  In Corollary \ref{cor:primitiveBW} we show that the Berger-Wang formula holds on general closed cones in finite dimensional spaces if the family $\mathcal{A}$ satisfies a primitivity condition.  But it is an open question whether the Berger-Wang formula holds if we only assume the family is equicontinuous and bounded. 
\end{question}

\section{Continuity Properties} \label{sec:continuity}

For every $R \ge 0$, let $\Omega_R := \{x \in K : \|x\| \le R \}$. 

\begin{lemma} \label{lem:Hausdorff}
Let $K$ be a closed cone in a finite dimensional normed space $X$.  Suppose that $\mathcal{A}_k$ and $\mathcal{B}_k$ are sequences of compact families of continuous, homogeneous maps $f: K \rightarrow K$ such that $\mathcal{A}_k$ converges to $\mathcal{A}$ and $\mathcal{B}_k$ converges to $\mathcal{B}$ in the Hausdorff metric on $C(\Omega_1, X)$.  Then $\mathcal{A}_k \mathcal{B}_k$ converges to $\mathcal{A} \mathcal{B}$ in the Hausdorff metric.  
\end{lemma}

\begin{proof}
There is an upper bound $M > 0$ such that $g(\Omega_1) \subseteq \Omega_M$ for all $g \in \mathcal{B} \cup \bigcup_k \mathcal{B}_k$.   
Choose $\epsilon> 0$. For all $f \in \mathcal{A}$ there exists $\delta > 0$ such that 
$$\|f(x) - f(y) \| \le \tfrac{1}{2} \epsilon$$ 
when $x, y \in \Omega_M$ have $\|x - y \| \le \delta$ since $\mathcal{A}$ is uniformly equicontinuous on the compact set $\Omega_M$. 
%Then there exists $N > 0$ large enough so that the Hausdorff distances $d(\mathcal{A}_k, \mathcal{A}) \le \tfrac{1}{2}\epsilon$ and $d(\mathcal{B}_k, \mathcal{B}) \le \tfrac{1}{2}\epsilon$ for all $k \ge N$. Therefore, 
%any fixed $f \in \mathcal{A}$ and $g \in \mathcal{B}$ we can choose $f_k \in \mathcal{A}_k$ and $g_k \in \mathcal{B}_k$ such that, we can choose $\delta >0$ small enough so that 
Choose $k$ large enough so that the Hausdorff distances $d(\mathcal{A}_k, \mathcal{A}) \le \tfrac{1}{2} \epsilon$ and  $d(\mathcal{B}_k, \mathcal{B}) \le \delta$.  Then, for any $f \in \mathcal{A}$, $g \in \mathcal{B}$, there exists $f_k \in \mathcal{A}_k$ and $g_k \in \mathcal{B}_k$ such that the following inequalities hold for all $x \in \Omega_M$ and $y \in \Omega_1$.  
$$\|f(x) - f_k(x) \| \le \tfrac{1}{2} \epsilon$$ 
and 
$$\|g(y) - g_k(y) \| \le \delta.$$ 
In addition, for any $f_k \in \mathcal{A}_k$ and $g_k \in \mathcal{B}_k$ there exist $f \in \mathcal{A}, g \in \mathcal{B}$ such that the inequalities above also hold for all $x \in \Omega_M$ and $y \in \Omega_1$.
Then we have 
\begin{align*}
\|f(g(x)) - f_k(g_k(x)) \| &\le \|f(g(x)) - f(g_k(x))\| + \|f(g_k(x)) - f_k(g_k(x))\|  \\
&\le \tfrac{1}{2} \epsilon + \tfrac{1}{2} \epsilon = \epsilon
\end{align*}
for all $x \in \Omega_1$.  Therefore for any fixed $f \in \mathcal{A}, g \in \mathcal{B}$,  
$$\inf \{ \|f\circ g - f_k \circ g_k \| : f_k \in  \mathcal{A}_k, g_k \in \mathcal{B}_k\} \le \epsilon.$$ 
Likewise, for any fixed $f_k \in \mathcal{A}_k, g_k \in \mathcal{B}_k$ 
$$\inf \{ \|f\circ g - f_k \circ g_k \| : f \in  \mathcal{A}, g \in \mathcal{B}\} \le \epsilon.$$  
So we conclude that the Hausdorff distance $d(\mathcal{A}_k\mathcal{B}_k, \mathcal{AB}) \le \epsilon$.  
\end{proof}

An immediate consequence of Lemma \ref{lem:Hausdorff} is that if $\mathcal{A}_k$ is a sequence of compact families of continuous, homogeneous maps $f$ on a closed cone $K$ that converges to a family $\mathcal{A}$ in the Hausdorff metric on $C(\Omega_1, X)$, then for any $m \in \N$, $\mathcal{A}_k^m$ converges to $\mathcal{A}^m$. Before proving the main result of this section, we need one other minor technical lemma. 

\begin{lemma} \label{lem:supCont}
Let $K$ be a closed cone in a finite dimensional normed space $X$. If $h:D \rightarrow \R$ is a continuous function on a closed subset $D \subseteq C(\Omega_1, X)$, then the map $\mathcal{A} \mapsto  \sup_{f \in \mathcal{A}} h(f)$ is continuous with respect to the Hausdorff metric on $D$. 
\end{lemma}

\begin{proof}
Suppose that $\mathcal{A}_k$ and $\mathcal{A}$ are compact subsets of $D$ such that $\mathcal{A}_k \rightarrow \mathcal{A}$ in the Hausdorff metric. Let 
$$\alpha_k := \sup_{f \in \mathcal{A}_k} h(f) \text{ and } \alpha := \sup_{f \in \mathcal{A}} h(f).$$
Since $\mathcal{A}$ is compact, there exists $f \in \mathcal{A}$ such that $h(f)$ is maximal.  Consider a sequence $f_k \in \mathcal{A}_k$ such that $f_k \rightarrow f$.  Then $h(f_k) \rightarrow h(f) = \alpha$.  Since $\alpha_k \ge h(f_k)$ for every $k$, it follows that $\liminf_{k \rightarrow \infty} \alpha_k \ge \alpha$.  

There is a subsequence $k_n$ such that $\lim_{n \rightarrow \infty} \alpha_{k_n} = \limsup_{k \rightarrow \infty} \alpha_k$. For each $k_n$, choose $f_n \in \mathcal{A}_{k_n}$ such that $\alpha_{k_n} = h(f_n)$.  Such an $f_n$ exists since each $\mathcal{A}_{k_n}$ is compact.   
Since $\mathcal{A}_{k_n} \rightarrow \mathcal{A}$ as $n \rightarrow \infty$, we can choose $g_n \in \mathcal{A}$ for each $n \in \N$ such that $\lim_{n \rightarrow \infty} \|f_n - g_n\| = 0$.  By passing to a subsequence if necessary, we can assume that $g_{n}$ converges to $g \in \mathcal{A}$. Then $\lim_{n \rightarrow \infty} h(f_{n}) = h(g)$.  Therefore 
$$\alpha \le \liminf_{k \rightarrow \infty} \alpha_k \le \limsup_{k \rightarrow \infty} \alpha_k = \alpha$$
which proves that $\lim_{k \rightarrow \infty} \alpha_k = \alpha$. 
\end{proof}

%Combining this observation with the Berger-Wang formula from Theorem \ref{thm:main}, we derive the following continuity property of the joint spectral radius.  
The main result of this section follows.

\begin{theorem}
Let $K$ be a closed polyhedral cone in a finite dimensional normed space $X$.  Suppose that $\mathcal{A}_k$ is a sequence of compact families of continuous, order-preserving, homogeneous maps $f: K \rightarrow K$ such that $\mathcal{A}_k$ converges to $\mathcal{A}$ in the Hausdorff metric on $C(\Omega_1, X)$.  Then $r(\mathcal{A}_k)$ converges to $r(\mathcal{A})$. 
\end{theorem}
\begin{proof}
%By Theorem \ref{thm:main}, 
%$$\sup_{m \in \N} \sup_{f \in \mathcal{A}^m} r(f)^{1/m} \le r(\mathcal{A}) \le \inf_{m \in \N} \sup_{f \in \mathcal{A}^m} \|f\|^{1/m}.$$
The spectral radius $r$ for continuous, order-preserving, homogeneous maps on polyhedral cones is continuous \cite[Corollary 4.5]{LemmensNussbaum13}. Therefore the map $\mathcal{A} \mapsto \sup_{f \in \mathcal{A}} r(f)$ is continuous with respect to the Hausdorff metric by Lemma \ref{lem:supCont}.  So is the map $\mathcal{A} \mapsto \sup_{f \in \mathcal{A}} \|f\|$. 

If follows from Lemma \ref{lem:Hausdorff} that $\mathcal{A} \mapsto \mathcal{A}^m$ is continuous in the Hausdorff metric for every $m \in \N$. Therefore both 
$$\mathcal{A} \mapsto \sup_{f \in \mathcal{A}^m} \|f\|^{1/m} \text{ and }\mathcal{A} \mapsto \sup_{f \in \mathcal{A}^m} r(f)^{1/m}$$
are continuous.  Since  
$$r(\mathcal{A}) =  \sup_{m \in \N} \sup_{f \in \mathcal{A}^m} r(f)^{1/m}$$
is a supremum of continuous functions, it is lower semicontinuous.  At the same time, $r(\mathcal{A}) = \hat{r}(\mathcal{A})$ by Theorem \ref{thm:main}. Therefore 
$$r(\mathcal{A}) = \hat{r}(\mathcal{A}) = \inf_{m \in \N} \sup_{f \in \mathcal{A}^m} \|f\|^{1/m}$$
is upper semicontinuous since it is the infimum of continuous function.  Therefore $r(\mathcal{A})$ is continuous in the Hausdorff metric. 
\end{proof}

\section{Boundedness Conditions} \label{sec:bounded}

\begin{definition} \label{def:irred}
Let $\mathcal{A}$ be a bounded family of continuous, order-preserving, homogeneous maps on a solid closed cone $K$ in a real normed space. We say that $\mathcal{A}$ is \emph{irreducible} if there is no non-trivial closed face of $K$ that is an invariant set for all $f \in \mathcal{A}$.  
%$$F_J := \{x \in \R^n_{ \ge 0} : x_j = 0 \text{ for all } j \notin J \}$$
%is an invariant set for all $f \in \mathcal{A}$.  
We say that $\mathcal{A}$ is \emph{primitive} if for every nonzero $x \in K$, there exists $f \in \mathcal{A}^+$ such that $f(x) \in \inter K$.  
\end{definition}

Note that \cite{DeiddaGuglielmiTudisco25} proposed a different definition of irreducibility for families of continuous, order-preserving, homogeneous maps on polyhedral cones. A family is irreducible in their terminology if and only if it is primitive in ours. 

The following result characterizes irreducibility for subadditive families on $\R^n_{\ge 0}$. 

\begin{lemma} \label{lem:irred}
Let $\mathcal{A}$ be a family of subadditive, order-preserving, homogeneous functions on $\R^n_{\ge 0}$. Then $\mathcal{A}$ is irreducible if and only if for every pair of distinct $i, j \in [n]$, there exists $m \in \N$ and $f \in \mathcal{A}^m$ such that $f(e_j)_i > 0$.  
\end{lemma}

\begin{proof}
Suppose that $\mathcal{A}$ is irreducible. 
Let $G(\mathcal{A})$ be the directed graph with vertices $[n]$, and an edge from $i$ to $j$ if there exists $f \in \mathcal{A}$ such that $f(e_j)_i > 0$.  Suppose that $G(\mathcal{A})$ is not strongly connected. Then there is a nonempty, proper set of vertices $I \subset [n]$ with no outbound edges in $G(\mathcal{A})$, that is, $f(e_j)_i = 0$ for all $i \in I$ and $j \in J := [n] \bs I$. Since the maps $f \in \mathcal{A}$ are subadditive, the closed face 
$$F_J := \{x \in \R^n_{\ge 0} : x_i = 0 \text{ for all } i \notin J \}$$
is invariant under $\mathcal{A}$.  This contradicts the assumption that $\mathcal{A}$ is an irreducible family.  So we conclude that the graph $G(\mathcal{A})$ is strongly connected.  

Since $G(\mathcal{A})$ is strongly connected, there is a path of some length $m < n$ that connects any $i$ to any $j$ in $[n]$.  Hence there is an $f \in \mathcal{A}^m$ such that $f(e_j)_i > 0$.  

Conversely, it is clear that no proper closed face $F_J$ can be invariant under $\mathcal{A}$ if there is an $f \in \mathcal{A}^m$ such that $f(e_j)_i > 0$ for any $j \in J$ and $i \notin J$.  
\end{proof}

The following theorem for subadditive families is similar to Elsner's theorem \cite[Lemma 4]{Elsner95} for families of matrices in $\C^{n \times n}$.

%\begin{theorem} \label{thm:subadditiveElsner}
%Let $K$ be a closed polyhedral cone in a Banach space $X$.  If $\mathcal{A}$ is a bounded family of  subadditive, continuous, order-preserving, homogeneous maps $f: K \rightarrow K$ and $r(\mathcal{A}) = 1$, then either the semigroup $\mathcal{A}^+$ is bounded or there is a proper face $F \ne \{0 \}$ of $K$ that is invariant under $\mathcal{A}$, that is, $f(F) \subseteq F$ for all $f \in \mathcal{A}$.  
%\end{theorem}

%\begin{theorem} \label{thm:subadditiveElsner}
%If $\mathcal{A}$ is a bounded family of  subadditive, continuous, order-preserving, homogeneous maps $f: \R^n_{\ge 0} \rightarrow \R^n_{\ge 0}$ and $r(\mathcal{A}) = 1$, then either the semigroup $\mathcal{A}^+$ is bounded or there is a proper face $F \ne \{0 \}$ of $\R^n_{\ge 0}$ that is invariant under $\mathcal{A}$, that is, $f(F) \subseteq F$ for all $f \in \mathcal{A}$.  
%\end{theorem}

\begin{theorem} \label{thm:subadditiveElsner}
If $\mathcal{A}$ is a bounded, irreducible family of subadditive, order-preserving, homogeneous maps on $\R^n_{\ge 0}$ and $r(\mathcal{A}) = 1$, then the semigroup $\mathcal{A}^+$ is bounded.
\end{theorem}

\begin{proof}
Suppose that $\{f(\mathbf{1}) : f \in \mathcal{A}^+ \}$ is unbounded. Then there exists $i \in [n]$ such that $\{f(\mathbf{1})_i : f \in \mathcal{A}^+ \}$ is unbounded.  By subadditivitiy, 
$$f(\mathbf{1}) \le \sum_{j \in [n]} f(e_j)$$
for all $f \in \mathcal{A}^+$.  Therefore there is a $j \in [n]$ such that $\{f(e_j)_i : f \in \mathcal{A}^+ \}$ is unbounded. Lemma \ref{lem:irred} implies that there is a $g \in \mathcal{A}^+$ such that $g(e_i)_j > 0$. Therefore the set $\{f(g(e_i))_i : f \in \mathcal{A}^+ \}$ is unbounded, which would imply that there is a map $f \circ g \in \mathcal{A}$ such that $f(g(e_i)) \ge c e_i$ where $c > 1$.  That would contradict $r(\mathcal{A}) = 1$, by Lemma \ref{lem:alpha}. Therefore we conclude that $\mathcal{A}^+$ is bounded.   
\end{proof}

The conclusion of Theorem \ref{thm:subadditiveElsner} can fail without the subadditivity assumption.  
\begin{example}
Let $f, g, h: \R^2_{\ge 0} \rightarrow \R^2_{\ge 0}$ be defined as follows.  
$$f(x) = \begin{bmatrix} x_1 + \min(x_1, x_2) \\ x_2 \end{bmatrix}, ~ g(x) =  \begin{bmatrix} 0 \\ x_1 + x_2 \end{bmatrix}, ~ h(x) = \begin{bmatrix} x_1+ x_2 \\ 0 \end{bmatrix}.$$
Let $\mathcal{A} = \{f,g,h \}$. It's clear that the iterates $f^k(\mathbf{1})$ are unbounded, however no non-trivial face of $\R^2_{\ge 0}$ is left invariant by all three maps. Observe the following composition relations:
$$\begin{array}{ccc}
f \circ g = g, & g \circ g = g, & h \circ g = h,  \\
f \circ h = h, & g \circ h = g, & h \circ h = h.
\end{array}$$
Therefore $\mathcal{A}^* = \{f^k, g \circ f^k, h \circ f^k : k \ge 0 \}$. By inspection, the only eigenvectors of $f^k$ (up to scaling) are $e_1$ and $e_2$, both with eigenvalue 1. The only eigenvectors of $g \circ f^k$ and $h \circ f^k$ are $e_2$ and $e_1$ respectively, again both with eigenvalue 1.  Therefore $r(f^k) = r(g \circ f^k) = r(h \circ f^k) = 1$ for all $k \in \N$. It follows that $r(\mathcal{A}) = 1$. 
\end{example}

Without subadditivity, we need to assume a stronger condition like primitivity to guarantee that $\mathcal{A}^+$ is bounded when $r(\mathcal{A}) = 1$. The following result holds for any closed cone with nonempty interior in a finite dimensional normed space.
 
\begin{theorem} \label{thm:primitiveElsner}
Let $K$ be a solid closed cone in a finite dimensional real normed space. If $\mathcal{A}$ is a bounded, primitive family of continuous, order-preserving, homogeneous maps on $K$ with $r(\mathcal{A}) = 1$, then the semigroup $\mathcal{A}^+$ is bounded.  
\end{theorem}
\begin{proof}
Let $\Sigma := \{x \in K : \|x\| = 1 \}$.  Choose any $u \in \inter K$.  
Suppose that $\mathcal{A}^+$ is unbounded.  Then for every $M > 0$, there exists $g \in \mathcal{A}^+$ such that $\|g(u)\| \ge M$. 

For any $f \in \mathcal{A}^+$ and $\epsilon > 0$, let $U_{f, \epsilon} := \{x \in K : f(x) \gg \epsilon u \}$.  Since $\mathcal{A}$ is a primitive family, every $x \in \Sigma$ is contained in at least one $U_{f, \epsilon}$.  Therefore the sets $U_{f, \epsilon}$ are an open cover for $\Sigma$.  Since $\Sigma$ is compact, there is a finite subcover.  Therefore there is one $\epsilon > 0$ such that for every $x \in \Sigma$, there exists $f \in \mathcal{A}^+$ with $f(x) \ge \epsilon u$.  In particular, there exists $f \in \mathcal{A}^+$ such that 
$$f\left(\frac{g(u)}{\|g(u)\|} \right) \ge \epsilon u.$$
It follows that 
$$f(g(u))  \ge  \epsilon u \|g(u)\| \ge \epsilon M u.$$
Since $M$ can be chosen arbitrarily large, we can guarantee that $\epsilon M > 1$, but then $r(\mathcal{A}) > 1$ by Lemma \ref{lem:alpha}, which is a contradiction.  
%Suppose that $\mathcal{A}^+$ is unbounded. Then there is at least one $i \in [n]$ such that $\{f(\mathbf{1})_i : f \in \mathcal{A}^+\}$ is unbounded.  Therefore, for every $\beta > 0$, there exists $f \in \mathcal{A}^+$ such that $f(\mathbf{1}) \ge \beta e_i$.  By primitivity, there exists $g \in \mathcal{A}^+$ and $\alpha > 0$ such that $g(e_i) \ge \alpha \mathbf{1}$.  Therefore $f(g(e_i)) \ge \alpha \beta e_i$.  Since $\beta$ can be chosen arbitrarily large, we can guarantee that $\alpha \beta > 1$, but then $r(\mathcal{A}) > 1$ by Lemmas \ref{lem:ineq} and \ref{lem:alpha}, which is a contradiction.  
\end{proof}

As an immediate consequence, we get the Berger-Wang formula for primitive families. 

\begin{corollary} \label{cor:primitiveBW}
Let $K$ be a solid closed cone in a finite dimensional real normed space. If $\mathcal{A}$ is a bounded, primitive family of continuous, order-preserving, homogeneous maps on $K$, then 
$$r(\mathcal{A}) = \hat{r}(\mathcal{A}).$$ 
\end{corollary}

\begin{proof}
We can assume that $r(\mathcal{A}) = 1$ by scaling the maps in $\mathcal{A}$. Then $\mathcal{A}^+$ is bounded by Theorem \ref{thm:primitiveElsner}.  Therefore $\hat{r}(\mathcal{A}) \le 1$ by Lemma \ref{lem:joint}.  So by Lemma \ref{lem:ineq} we have 
$$1 = r(\mathcal{A}) \le \hat{r}(\mathcal{A}) \le 1.$$
Therefore $r(\mathcal{A}) = \hat{r}(\mathcal{A})$. 
\end{proof}

\section{Extremal and Barabanov Norms} \label{sec:norms}
%\hl{Say more about Barabanov here...}
An \emph{extremal norm} for a family of functions $\mathcal{A}$ on $\R^n_{\ge 0}$ is a norm $\|\cdot \|_*$ such that 
$$\|f(x) \|_* \le \hat{r}(\mathcal{A}) \|x\|_*$$
for all $x \in \R^n_{\ge 0}$ and $f \in \mathcal{A}$.  An extremal norm for $\mathcal{A}$ is called a \emph{Barabanov norm} if for every $x \in \R^n_{\ge 0}$, there exists $f \in \mathcal{A}$ such that 
$$\|f(x)\|_* = \hat{r}(\mathcal{A}) \|x\|_*.$$

For any $x \in \R^n$, let $|x| \in \R^n_{\ge 0}$ be the vector obtained by applying the absolute value function to each entry of $x$. 
If $\nu: \R^n_{\ge 0} \rightarrow \R$ is any order-preserving, subadditive, homogeneous function such that $\nu(x) > 0$ for all non-zero $x \in \R^n_{\ge 0}$, then observe that the function $x \mapsto \nu(|x|)$ is a norm on $\R^n$.  

\begin{theorem} \label{thm:extremal}
Let $\mathcal{A}$ be a bounded irreducible family of subadditive, order-preserving, homogeneous maps on $\R^n_{\ge 0}$. Let $\|\cdot \|$ be a monotone norm on $\R^n$.  Then the function 
$$\|x \|_* := \sup_{f \in \mathcal{A}^*} \| f(|x|) \|$$
is a monotone extremal norm. 
\end{theorem}

\begin{proof}
We may assume without loss of generality that $\hat{r}(\mathcal{A}) = 1$ by scaling the functions in $\mathcal{A}$ if necessary.  Note that $\hat{r}(\mathcal{A}) = r(\mathcal{A})$ by Theorem \ref{thm:main} and Lemma \ref{lem:subEqui}.  Therefore $\|x\|_*$ is finite for all $x \in \R^n$ by Theorem \ref{thm:subadditiveElsner}.

From the definition we have 
$$\|g(x) \|_* = \sup_{f \in \mathcal{A}^+} \| f(g(x)) \| \le \sup_{f \in \mathcal{A}^+} \| f(x) \| = \|x\|_*$$
for all $g \in \mathcal{A}$ and $x \in \R^n_{\ge 0}$.  So we just need to verify that $\|\cdot \|_*$ is a norm.  It is clear that $\|\cdot \|_*$ is homogeneous, nonnegative, and monotone. Since $f$ is order-preserving and subadditive, $f(|x + y|) \le f(|x|) + f(|y|)$ for any $x,y \in \R^n_{\ge 0}$.  Then since $\|\cdot \|$ is monotone, we have
\begin{align*}
\|x+y\|_* &= \sup_{f \in \mathcal{A}^+} \| f(|x + y|) \| \\
%&\le \sup_{f \in \mathcal{A}^+} \| f(|x|)  + f(|y|) \| \\
&\le \sup_{f \in \mathcal{A}^+} \| f(|x|) \| + \|f(|y|) \| \\
&\le \sup_{f \in \mathcal{A}^+} \| f(|x|) \| + \sup_{g \in \mathcal{A}^+} \|g(|y|) \| \\
&= \|x\|_* + \|y\|_*.
\end{align*}
Finally, since the identity map is included in $\mathcal{A}^*$, observe that $\|x\|_* \ge \|x\| > 0$ for all nonzero $x \in \R^n$.
%Finally, observe that if $\|x\|_* = 0$, then $f(|x|) = 0$ for every $f \in \mathcal{A}$.  Let $P_x$ be the part of $\R^n_{\ge 0}$ containing $|x|$.  Then the closure of $P_x$ is invariant for every $f \in \mathcal{A}$. Since the closure of $P_x$ is a closed face of $K$, it follows by irreducibility that $P_x$ must be trivial and $x = 0$.  
\end{proof}

\begin{theorem} \label{thm:Barabanov}
Let $\mathcal{A}$ be a compact irreducible family of subadditive, order-preserving, homogeneous maps on $\R^n_{\ge 0}$. If $\|\cdot\|_*$ is a monotone extremal norm for $\mathcal{A}$, then the function 
$$\|x\|_{**} := \lim_{m \rightarrow \infty} \sup_{f \in \mathcal{A}^m} \| f(|x|) \|_*$$
is a Barabanov norm for $\mathcal{A}$. 
\end{theorem} 

\begin{proof}
We may assume without loss of generality that $\hat{r}(\mathcal{A}) = 1$, by scaling the functions in $\mathcal{A}$ if necessary.  
Fix $x \in \R^n$ and let $\alpha_m = \sup_{f \in \mathcal{A}^m} \| f(|x|) \|_*$ for each $m \in \N$.  Since $\|\cdot \|_*$ is extremal, it follows that $\alpha_m$ is a non-increasing sequence, and therefore the limit defining $\|x\|_{**}$ exists for every $x \in \R^n$.  

It is easy to see that $\|\lambda x\|_{**} = |\lambda| \|x\|_{**}$ for all $\lambda \in \R$ and that $\|x\|_{**}$ is always nonnegative.  Since $\|\cdot\|_*$ is monotone and any  $f \in \mathcal{A}^+$ is order-preserving and subadditive,  we have
\begin{align*}
\sup_{f \in \mathcal{A}^m} \| f(|x + y|) \|_* &\le \sup_{f \in \mathcal{A}^m} \| f(|x| + |y|) \|_*  \\
&\le \sup_{f \in \mathcal{A}^m} \| f(|x|) \|_* + \|f(|y|) \|_* \\
&\le \sup_{f \in \mathcal{A}^m} \| f(|x|) \| + \sup_{g \in \mathcal{A}^m} \|g(|y|) \| \\
\end{align*}
for all $m \in \N$.  It follows by taking the limit as $m$ goes to infinity that $\|x+y\|_{**} \le \|x\|_{**} + \|y\|_{**}$ for all $x, y \in \R^n$.  

The function $\|\cdot\|_{**}$ also has the extremal property.  For every $g \in \mathcal{A}$ and $x \in \R^n_{\ge 0}$, 
$$\|g(x) \|_{**} = \lim_{m \rightarrow \infty} \sup_{f \in \mathcal{A}^m} \|f(g(x))\|_* \le \lim_{m \rightarrow \infty} \sup_{f \in \mathcal{A}^{m+1}} \|f(x)\|_* = \|x\|_{**}.$$

We now show that $\|\cdot \|_{**}$ has the Barabanov property.  
For each $n \in \N$, there exists $f_n \in \mathcal{A}^n$ such that $\|f_n(|x|)\|_* \ge \alpha_n - \tfrac{1}{n}$ where $\alpha_n := \sup_{f \in \mathcal{A}^n} \|f(|x|)\|_*$. Each $f_n$ is a composition of functions $f_{n,n}\circ \cdots \circ f_{n,2}\circ f_{n,1}$ where each $f_{n, k} \in \mathcal{A}$.  We can choose an increasing sequence $m_n$ such that $g_1 = \lim_{m_n \rightarrow \infty} f_{m_n,1}$ exists in $\mathcal{A}$.  Then for each $k \in \N$, we can take a refinement of the sequence $m_n$ and assume that $g_k = \lim_{n \rightarrow \infty} f_{m_n,k}\circ \cdots \circ f_{m_n,2}\circ f_{m_n,1}$ exists in $\mathcal{A}^k$.  Observe that for every $k$ and $m_n \ge k$,
$$\|f_{m_n,1}(|x|)\|_* \ge \|f_{m_n,k}\circ \cdots \circ f_{m_n,2}\circ f_{m_n,1} ( |x| )\|_* \ge \|f_{m_n}(|x|)\|_* \ge \alpha_{m_n} - \tfrac{1}{m_n}$$
since $\|\cdot\|_*$ has the extremal property.  Passing to the limit as $m_n$ approaches infinity, we get 
$$\|g_1 ( |x| )\|_* \ge \|g_k(|x|)\|_* \ge \|x\|_{**}$$
for every $k$.  If we let $k \rightarrow \infty$, then we see that 
$$\|g_1(|x|)\|_{**} = \lim_{k \rightarrow \infty} \sup_{f \in \mathcal{A}^{k-1}} \|f(g_1 ( |x| ))\|_* \ge \|x\|_{**}.$$
We already know that $\|g_1(|x|)\|_{**} \le \|x\|_{**}$ since $\|\cdot\|_{**}$ also has the extremal property, therefore $\|g_1(|x|)\|_{**} = \|x\|_{**}$ and we conclude that $\|\cdot \|_{**}$ has the Barabanov property. 

It remains to show that $\|x\|_{**} = 0$ if and only if $x = 0$.  Note that the norm used to define the joint spectral radius $\hat{r}(\mathcal{A})$ is arbitrary.  Therefore we could use the extremal norm $\|\cdot \|_*$ and observe that 
$$\hat{r}(\mathcal{A}) = 1 = \lim_{m \rightarrow \infty} \sup_{f \in \mathcal{A}^m} \|f\|_*^{1/m}.$$
Let $\beta_m := \sup_{f \in \mathcal{A}^m} \|f\|_*$ for each $m \in \N$. We know that $\beta_m \le 1$ for all $m$ since $\|f\|_* \le \hat{r}(\mathcal{A})^m = 1$ for all $f \in \mathcal{A}^m$.  As noted in the proof of Lemma \ref{lem:joint}, $\beta_{m+n} \le \beta_m \beta_n$ for all $m, n \in \N$.   Suppose that $\beta_m < 1$ for some $m \in \N$.  Then it follows that $\beta_{mk} \le \beta_m^k$ for all $k \in \N$, and therefore 
$$\hat{r}(\mathcal{A}) = \lim_{k \rightarrow \infty} \beta_{mk}^{1/mk} \le \beta_m^{1/m} < 1$$
which is a contradiction since $\hat{r}(\mathcal{A}) = 1$.  Therefore we conclude that $\beta_m = 1$ for all $m \in \N$.  Since $\mathcal{A}$ is compact, there exists $f \in \mathcal{A}^m$ such that $\|f\|_* = 1$ for every $m \in \N$. Let $\Sigma := \{x \in \R^n_{\ge 0} : \|x\|_* = 1 \}$ and for each $m \in \N$, let
$$\Sigma_m := \{ x \in \Sigma : \text{ there exists } f \in \mathcal{A}^m \text{ with } \|f(x)\|_*  = 1 \}.$$
The sets $\Sigma_m$ form a nested sequence of compact subsets of $\Sigma$, therefore, they have a nonempty intersection $\Sigma_\infty$.  For any $x \in \Sigma_\infty$, $\|x\|_{**} = 1$.  

Since we have already shown that $\|\cdot \|_{**}$ is a subadditive function, we have
$$\|x\|_{**} \le \sum_{i \in [n]} x_i \|e_i\|_{**}.$$ 
Therefore, at least one $e_i$ has $\|e_i\|_{**} > 0$. 

Consider any nonzero $x \in \R^n$.  Then at least one entry $j$ of $|x|$ is strictly positive, so there exists $\epsilon > 0$ such that $|x| \ge \epsilon e_j$.  Since $\mathcal{A}$ is an irreducible family, there is a $g \in \mathcal{A}^+$ such that $g(e_j)_i > \delta e_i$ for some $\delta > 0$ by Lemma \ref{lem:irred}.  Therefore, 
$$\|x\|_{**} \ge \epsilon \|e_j\|_{**} \ge \epsilon \|g(e_j)\|_{**} \ge \epsilon \delta \|e_i\|_{**} > 0.$$

\end{proof}

\section{Generalized and joint spectral subradii} \label{sec:subradii}

The quantities generalized spectral subradius  and joint spectral subradius for sets of matrices have been introduced in
\cite{gurvitz95} to present conditions for Markov asymptotic stability of discrete linear inclusion and are also related to the so-called mortality problem \cite{czornik05}. It was shown in \cite{gurvitz95} that these quantities are actually the same for finite sets of matrices and this results was extended to nonempty sets of matrices in \cite{czornik05}. In this section we extend these results to wedges (and build on the setting from \cite{MP17}, \cite{MP18}).

Let $X$ be real normed vector space and $K\subset X$ a non-zero wedge. Assume  
 $\mathcal{A}$ is a nonempty set of homogeneous and bounded maps on $K$ and let 
 $\mathcal{A}^m=\{f_1 \circ \cdots \circ f_m: f_i \in \mathcal{A}, i=1, \ldots , m\}$. 
 Let $\beta_m = \inf_{f \in \mathcal{A}^m} \|f\|$ and  $\gamma_m = \inf_{f \in \mathcal{A}^m} r(f)$.
  The joint spectral subradius $\hat{r}_* (\mathcal{A})$  and the generalized spectral subradius $\overline{r}_* (\mathcal{A})$ of $\mathcal{A}$ are defined as
  $$\hat{r}_* (\mathcal{A})= \inf _{m\in \mathbb{N}} \beta _m ^{1/m}$$
  and
    $$\overline{r}_* (\mathcal{A})= \inf _{m\in \mathbb{N}} \gamma _m ^{1/m},$$
 respectively. It obviously holds that $\overline{r}_* (\mathcal{A})\le \hat{r}_* (\mathcal{A})$ since $r(f) \le \|f\|$ holds for each homogeneous and bounded map $f$ on $K$.
 
 \begin{remark}{\rm Observe that our results below apply also to the interesting special case when $K=X$ is a Banach algebra. 
  }
 \end{remark}
 
 The following four theorems generalize \cite[Theorem 1-4]{czornik05} and are provable in very similar way as these results. The proofs are included for the sake of completeness.
 
 \begin{theorem}\label{cz1}
 Let $X$ be real normed vector space, $K\subset X$ a non-zero wedge and let  $\mathcal{A}$ be a nonempty set of homogeneous and bounded maps on $K$. Then $\overline{r}_* (\mathcal{A})= \hat{r}_* (\mathcal{A})$.
 \end{theorem}
 \begin{proof} To prove the nontrivial inequality assume $\varepsilon >0$. By definition of $\overline{r}_* (\mathcal{A})$ there exists $N$ such that
 $$(\overline{r}_* (\mathcal{A}) + \varepsilon)^N > \gamma _N.$$
 Thus there exists $f_{1}, \ldots , f_N \in  \mathcal{A}$ such that 
  \begin{equation}
  \overline{r}_* (\mathcal{A}) + \varepsilon > r(f_{1} \circ \cdots \circ f_N)^{\frac{1}{N}}.
  \label{eps}
  \end{equation}
   It follows from definition of $\hat{r}_* (\mathcal{A})$ that for each $k \in \mathbb{N}$ we have
  $$\hat{r}_* (\mathcal{A}) \le \beta _{Nk} ^{\frac{1}{Nk}} \le \|(f_{1} \circ \cdots \circ f_N)^k\|^{\frac{1}{Nk}}.$$
  Letting $k$ to infinity it follows that 
  $$\hat{r}_* (\mathcal{A}) \le r(f_{1} \circ \cdots \circ f_N)^{\frac{1}{N}},$$
  which in combination with (\ref{eps}) implies that
  $\overline{r}_* (\mathcal{A})\ge \hat{r}_* (\mathcal{A})$, which completes the proof.
 \end{proof}
 
 Let us,  under the assumptions of Theorem \ref{cz1}, denote $r_* (\mathcal{A})= \overline{r}_* (\mathcal{A})= \hat{r}_* (\mathcal{A})$.
 
 \begin{theorem}\label{cz2}
 Let $X$ be real normed vector space, $K\subset X$ a non-zero wedge and let  $\mathcal{A}$ be a nonempty set of homogeneous and bounded maps on $K$. Then 
$$r_* (\mathcal{A})= \lim _{m\to \infty} \beta _m ^{\frac{1}{m}}= \liminf _{m\to \infty} \gamma _m ^{\frac{1}{m}}.$$
 \end{theorem}
 \begin{proof} Since $\beta _{m +n}\le\beta _{m}\beta _{n}$ the $\lim _{m\to \infty} \beta _m ^{\frac{1}{m}}$ exists by Fekete's lemma.  By definition it is clear that 
 $\hat{r}_* (\mathcal{A})\le \lim _{m\to \infty} \beta _m ^{\frac{1}{m}}$. To prove the reverse inequality we  may by \cite[Lemma 1]{czornik05} assume that there exists $N$ such that 
  $\hat{r}_* (\mathcal{A})=\beta _N ^{1/N}$.  Let $\varepsilon >0$. Then there exist
  $f_{1}, \ldots , f_N \in  \mathcal{A}$ such that
  $$\hat{r}_* (\mathcal{A}) + \varepsilon > \|f_{1}\circ \cdots \circ f_N\|^{\frac{1}{N}}.$$
  It follows that for each $k$ we have
  $$ \beta _{Nk} ^{\frac{1}{Nk}} \le \|(f_{1}\circ  \cdots \circ f_N)^k\|^{\frac{1}{Nk}} \le \|f_{1}\circ \cdots \circ f_N\|^{\frac{1}{N}} < \hat{r}_* (\mathcal{A}) + \varepsilon$$
  and so $ \lim _{m\to \infty} \beta _m ^{\frac{1}{m}}\le \hat{r}_* (\mathcal{A}) $.
  
  It remains to prove that $\overline{r}_* (\mathcal{A})=\liminf _{m\to \infty} \gamma _m ^{\frac{1}{m}}$. Clearly $\overline{r}_* (\mathcal{A}) \le \liminf _{m\to \infty} \gamma _m ^{\frac{1}{m}}$, so we need to prove the reverse inequality. Again by \cite[Lemma 1]{czornik05} we may assume that there exists $N$ such that 
  $\overline{r}_* (\mathcal{A})=\gamma _N ^{1/N}$.  Let $\varepsilon >0$. Then there exist
  $f_{1}, \ldots , f_N \in  \mathcal{A}$ such that for  each $k$ we have
  $$\overline{r}_* (\mathcal{A}) + \varepsilon > r(f_{1}\circ \cdots \circ f_N)^{\frac{1}{N}}=r((f_{1}\circ \cdots \circ f_N)^k)^{\frac{1}{Nk}}\ge \gamma _{Nk} ^{\frac{1}{Nk}}.$$
  This implies $\overline{r}_* (\mathcal{A}) \le \liminf _{m\to \infty} \gamma _m ^{\frac{1}{m}}$, which completes the proof.
 \end{proof}
  Let $X$ be real normed vector space, $K\subset X$ a non-zero wedge and let  $\mathcal{A}$ be a nonempty set of homogeneous and bounded maps on $K$.
 Let $\mathcal{D}$ be the set of sequences $d=(f_1, f_2, f_3, \ldots)$  of maps from $\mathcal{A}$. For a fixed $d\in \mathcal{D}$ define a homogeneous and bounded map $F_m ^d$ on $K$ by
 $F_m ^d = f_m\circ f_{m-1} \circ \cdots \circ f_1$
 and $r (d)=\liminf _{m \to \infty} \|F_m ^d\|^{\frac{1}{m}}$. As pointed out in  \cite{czornik05} the quantity $ \log r (d)$ is in control theory literature (at least in the linear case) known as the greatest Lyapunov exponent of the time varying system $x(m+1)=f_m (x(m))$ where $f_m \in \mathcal{A}$. (Refer here to Mason, Gursoy for the tropical case and to the books Max plus at work and Bacelli et al)
 
\begin{theorem}\label{cz3}
 Let $X$ be real normed vector space, $K\subset X$ a non-zero wedge and let  $\mathcal{A}$ be a nonempty set of homogeneous and bounded maps on $K$. Then 
 $r_* (\mathcal{A})=\inf _{d\in \mathcal{D}} r(d)$.
 \end{theorem}
 \begin{proof} Since   $\|F_m ^d\|^{\frac{1}{m}}\ge  \beta _m ^{\frac{1}{m}}$ for each $m$ and each $d\in \mathcal{D}$, it follows that $r_* (\mathcal{A})\le\inf _{d\in \mathcal{D}} r(d)$. To prove the reverse inequality fix $N\in \mathbb{N}$ and $f_{i_1}, \ldots , f_{i_N} \in  \mathcal{A}$. Define a sequence $d_0 \in \mathcal{D}$ by starting with $f_{i_1}, \ldots , f_{i_N}$ and then continuing periodically.  Then
 $$\inf _{d\in \mathcal{D}} r(d) \le r(d_0) \le \liminf _{m \to \infty}\|(f_{i_N}\circ \cdots  \circ f_{i_1})^m\|^{\frac{1}{Nm}}\le\|f_{i_N}\circ \cdots  \circ f_{i_1}\|^{\frac{1}{N}}.$$
 Since $N\in \mathbb{N}$ and $f_{i_1}, \ldots , f_{i_N} \in  \mathcal{A}$ were arbitrary it follows that $r_* (\mathcal{A})\ge\inf _{d\in \mathcal{D}} r(d)$, which completes the proof.
 \end{proof}
  
\subsection{Selectable stability of discrete inclusions}
 
Let $DI(\mathcal{A})$ be a set of all sequences $\{x(m):m \in \mathbb{N}\} \subset K$ such that  
 $x(m+1)=f_m (x(m))$ for some $f_m \in \mathcal{A}$. We say that  $DI(\mathcal{A})$ is \emph{selectably stable}
 if and only if there exists a sequence $d\in \mathcal{D}$ such that $\lim _{m\to \infty} \|F_m ^d\|=0$.
  
 \begin{theorem}\label{cz4}
 Let $X$ be real normed vector space, $K\subset X$ a non-zero wedge and let  $\mathcal{A}$ be a nonempty set of homogeneous and bounded maps on $K$. Then  $DI(\mathcal{A})$ is selectably stable if and only if
 $r_* (\mathcal{A})<1$.
 \end{theorem}
 \begin{proof} If $DI(\mathcal{A})$ is selectably stable then there exists $d=(f_1, f_2, f_3, \ldots)\in \mathcal{D}$ such that $\lim _{m\to \infty} \|F_m ^d\|=0$. So there exists $N$ such that
 $\|f_N \circ \cdots \circ f_1\|<\frac{1}{2}$. Thus 
 $r_* (\mathcal{A})<(\frac{1}{2})^{\frac{1}{N}}<1$.
 
 Conversely, assume that $r_* (\mathcal{A})<1$. Then for a fixed $q \in(r_* (\mathcal{A}),1)$ there exists $N\in \mathbb{N}$ and $f_{i_1}, \ldots , f_{i_N} \in  \mathcal{A}$ such that $\|f_{i_N}\circ \cdots  \circ f_{i_1}\|<q$. Again let $d_0 \in \mathcal{D}$  be a sequence  starting with $f_{i_1}, \ldots , f_{i_N}$ and then continuing periodically. It follows that $\lim _{m\to \infty} \|F_m ^{d_0}\|=0$.
 \end{proof}
%Let $K$ be a closed cone in $X$ and let $\Sigma = \{x \in K : \|x\| = 1\}$. The norm of a continuous, homogeneous map $f:K \rightarrow K$ is $\|f\| = \sup \{\|f(x)\| : x \in \Sigma \}$. 
%The \emph{cone spectral radius} of a continuous, order-preserving, homogeneous map $f:K \rightarrow K$ is 
%$$r(f) = \lim_{k \rightarrow \infty} \|f^k\|^{1/k} = \inf_{k > 0} \|f^k\|^{1/k}.$$

\begin{remark}{\rm The last four theorems remain true (with the same proof) if $K\subset X$ is a non-zero set such that $tK \subset K $ for all $t\ge 0 $.
Thus they can be applied to  homogeneous and bounded maps defined on max cones studied in \cite{MP17} and \cite{MP18} and thus also to the tropical case studied in e.g. \cite{MalletParetNussbaum02, Lur06, Peperko08, LemmensNussbaum, GM11, BCOQ92, Bu10}.
  }
 \end{remark}
 
 \noindent {\bf Acknowledgements.}
The authors acknowledge  mobility support by the Slovenian Research and Innovation Agency (Slovenia-USA bilateral project BI-US/22-24-046).

 The second author acknowledges a partial support of  the Slovenian Research and Innovation Agency (grants P1-0222). %, J1-8133 

\bibliography{BW}
\bibliographystyle{plain}

\end{document}